\documentclass[a4paper,12pt]{amsart}
\usepackage[margin=2cm]{geometry}

\usepackage{amscd,amsmath,amsthm}

\usepackage{fouriernc} %font of cisinski-deglise:etale motives
\usepackage{MnSymbol} % This makes \mathcal as usual because fouriernc changes it

\usepackage{enumerate}
\usepackage{color}
\usepackage[pdfusetitle,unicode]{hyperref}
\usepackage[cmtip,all]{xy} % centering label of arrow via e.g. ar[r]^-{label} (the minus makes it centering)
\usepackage{enumitem}
\usepackage{datetime}
\usepackage{setspace}					% Zeilenabstand
\usepackage{microtype}					% macht Raender und Grauwert besser
\usepackage{graphicx}					% Graphik
\usepackage{epsfig}						% Graphik
\usepackage{verbatim}
\usepackage[english]{babel}		% neue Rechtschreibung
\usepackage[utf8]{inputenc} 			% Umlaute ermoeglichen
\usepackage{enumitem}					% Aufzaehlungen modifizieren
\usepackage{xspace} 					% setzt Leerzeichen, wenn welche hingehören (vor nächstem wort, aber nicht vor "." oder ")")
\hypersetup{
  linktoc=page
}

\setitemize{leftmargin=0.75cm}

\newcommand{\mathletter}[1]{\mathbf{#1}}
\newcommand{\N}{\mathletter{N}}				% natural numbers
\newcommand{\Z}{\mathletter{Z}}				% integral numbers
\newcommand{\Q}{\mathletter{Q}}				% rational numbers
\newcommand{\F}{\mathletter{F}}				% finite fields

\newcommand{\A}{\mathletter{A}}

\newcommand{\mfrak}{\mathfrak{m}}
\newcommand{\pfrak}{\mathfrak{p}}

\renewcommand{\to}[1][]{\overset{#1}{\rightarrow}}		
\newcommand{\To}[1][]{\overset{#1}{\longrightarrow}}	
\newcommand{\inj}[1][]{\overset{#1}{\hookrightarrow}}		% injective maps
\newcommand{\inv}{^{-1}}								% inverse
\newcommand{\skp}[1]{\langle #1\rangle}				% inner product, generation brackets

\DeclareMathOperator{\Ho}{H}					% homology
\DeclareMathOperator{\id}{id}				% id
\DeclareMathOperator{\Tor}{Tor}				% Tor
\DeclareMathOperator{\K}{K}						% K-theory
\DeclareMathOperator{\KH}{KH}					% KH
\DeclareMathOperator{\Ge}{G}						% G-theory
\newcommand{\KZR}{\K^\mathrm{RZ}}				% K-theory of ZR-space

\newcommand{\df}[1]{\textbf{#1}}					% definiendum in capital letters
\renewcommand{\phi}{\varphi}
\renewcommand{\epsilon}{\varepsilon}
\renewcommand{\rho}{\varrho}

\theoremstyle{definition}				% normal
	\newtheorem{defin}{Definition}[section]
	\newtheorem{example}[defin]{Example}
	\newtheorem{remark}[defin]{Remark}

	\newtheorem{question}[defin]{Question}
	
	\newtheorem*{notation*}{Notation}
	\newtheorem{notation}[defin]{Notation}
	
\theoremstyle{plain} 					% inner text italic
	\newtheorem{thm}[defin]{Theorem}
	\newtheorem*{thm*}{Theorem}
	\newtheorem*{obs*}{Observation}
	\newtheorem{prop}[defin]{Proposition}
	\newtheorem{lemma}[defin]{Lemma}
	\newtheorem{cor}[defin]{Corollary}
	\newtheorem*{cor*}{Corollary}

		\newcounter{zaehler}

\theoremstyle{remark} 					% name of environment italic, not bold

\newcounter{formulanumber}[defin]

\DeclareMathOperator*{\colim}{colim}

\newcommand{\mymathcal}[1]{\mathcal{#1}}

\renewcommand{\O}{\mathcal{O}}

\newcommand{\Mcal}{\mymathcal{M}}

\DeclareMathOperator{\RZ}{RZ}

\newcommand{\cdh}{\mathrm{cdh}}

\newcommand{\op}{^\mathrm{op}}

\newcommand{\mathcat}[1]{\mathrm{#1}}

\newcommand{\Vect}{\mathcat{Vec}}

\newcommand{\Mod}{\mathcat{Mod}}

\newcommand{\PCoh}{\mathcat{PCoh}}

\newcommand{\Mdf}{\mathcat{Mdf}}

\newcommand{\Open}{\mathcat{Open}}
\newcommand{\Model}{\mathcat{Model}}

\newcommand{\mathinfcat}[1]{\mathbf{#1}}

\newcommand{\Sp}{\mathinfcat{Sp}}
\newcommand{\Sh}{\mathinfcat{Sh}}

\newcommand{\Perf}{\mathinfcat{Perf}}

\newcommand{\RModinf}{\mathinfcat{RMod}}

\DeclareMathOperator{\Spec}{Spec}		% Spec
\newcommand{\carre}[4]{\[\begin{xy}\xymatrix{#1\ar[r]\ar[d]&#2\ar[d]\\#3\ar[r]&#4}\end{xy}\]}
\newcommand{\carrelift}[4]{\[\begin{xy}\xymatrix{#1\ar[r]\ar[d]&#2\ar[d]\\#3\ar[r]\ar@{.>}[ur]&#4}\end{xy}\]}
\newcommand{\carretag}[5][]{\[\tag{#1}\begin{xy}\xymatrix{#2\ar[r]\ar[d]&#3\ar[d]\\#4\ar[r]&#5}\end{xy}\]}
\newcommand{\carremap}[8]{\[\begin{xy}\xymatrix{#1\ar[r]^{#2}\ar[d]_{#3}&#4\ar[d]^{#5}\\#6\ar[r]_{#7}&#8}\end{xy}\]}

\newcommand{\lift}[3]{\[\begin{xy}\xymatrix{&#1\ar[d]\\#2\ar[r]\ar@{.>}[ur]&#3}\end{xy}\]}
\newcommand{\liftmap}[6]{\[\begin{xy}\xymatrix{&#1\ar[d]^{#2}\\#3\ar@{.>}[ur]^{#4}\ar[r]_{#5}&#6}\end{xy}\]}

\reversemarginpar		% Randnotiz rechts, da in amsart normeilerweise links.
\title{Regularity of semi-valuation rings and homotopy invariance of algebraic K-theory}
\author{Christian Dahlhausen}
\address{Institut für Mathematik, Universität Heidelberg, Im Neuenheimer Feld 205, 69120 Heidelberg, Germany}
\email{cdahlhausen@mathi.uni-heidelberg.de}
\urladdr{cdahlhausen.eu}
\hypersetup{
pdftitle={Regularity of semi valuation rings and homotopy invariance of algebraic K-theory},
pdfauthor={Christian Dahlhausen},
pdfkeywords={K-Theory, Regular rings, Algebraic Geometry},
pdfmenubar=true
}
\begin{document}
\begin{abstract}
We show that the algebraic K-theory of semi-valuation rings with stably coherent regular semi-fraction ring satisfies homotopy invariance. Moreover, we show that these rings are regular if their valuation is non-trivial. Thus they yield examples of regular rings which are not homotopy invariant for algebraic K-theory. On the other hand, they are not necessarily coherent, so that they provide a class of possibly non-coherent examples for homotopy invariance of algebraic K-theory. As an application, we show that Temkin's relative Riemann-Zariski spaces also satisfy homotopy invariance for K-theory under some finiteness assumption.
\end{abstract}
\maketitle
\setcounter{tocdepth}{1}
\tableofcontents

%%%%%%%%%%%%%%%%%%%%%%%%%%%%%%%%%%%%%%%%%%%%%%%%%%%%%%%%%%%%%%%%%%%%%%%
\section{Introduction}
%%%%%%%%%%%%%%%%%%%%%%%%%%%%%%%%%%%%%%%%%%%%%%%%%%%%%%%%%%%%%%%%%%%%%%%

Algebraic K-theory of rings is not $\A^1$-invariant in general, but homotopy invariance is known for some classes of rings, for instance:
\begin{enumerate}[leftmargin=36pt]
\item stably coherent regular rings (e.g.\ noetherian regular, valuation, or Prüfer rings)
\item perfect $\F_p$-algebras
\item certain rings arising from $C^*$-algebras
\item certain rings of continuous functions
\end{enumerate}
Here we use a very general notion of regularity generalising the usual one in the noetherian context, see Definition~\ref{regular--def}.
In this note, we expand this list by certain semi-valuation rings (Definition~\ref{semi-valuation-ring--def}) which can be non-coherent (Lemma~\ref{not-coherent--lem} and Lemma~\ref{not-coherent-2--lem}).

\begin{thm*}[Corollary~\ref{K-semi-valuation-ring--cor}]
Let $(A^+,\pfrak)$ be a semi-valuation ring whose semi-fraction ring $A=A^+_\pfrak$ is stably coherent and regular. Then the canonical maps
\begin{enumerate}
\item $\K_{\geq 0}(A^+) \To \K(A^+)$, 
\item $\K(A^+) \To \K(A^+[t_1,\ldots,t_k])$, and
\item $\K(A^+) \To \KH(A^+)$
\end{enumerate}
are equivalences. Here, $\K(A^+)$ denotes the non-connective algebraic K-theory spectrum of $A^+$ and $\KH(A^+)$ its homotopy invariant K-theory spectrum.
\end{thm*}

The relationship between regularity of rings and homotopy invariance of their algebraic K-theory has been studied since the origins of K-theory. 
The case of noetherian regular rings has been proved by Quillen \cite{quillen73} and the same proof works more generally for stably coherent regular rings, see Swan \cite{swan-coherent}. Another proof for stably coherent regular rings was given by Waldhausen \cite[Thm. 3 \& Thm. 4]{waldhausen-free-1}. For valuation rings there are also proofs by Kelly-Morrow \cite[Thm.~3.2]{kelly-morrow} and Kerz-Strunk-Tamme \cite[Lem.~4.3]{kst-vorst}. Banerjee-Sadhu showed that Prüfer domains are stably coherent and regular \cite{banerjee-sadhu}.
Recently, Antieau-Mathew-Morrow showed homotopy invariance for perfect $\F_p$-algebras \cite[Prop.~5.1]{antieau-mathew-morrow}.
For $C^*$-algebras, Higson has shown the statement for stable ones \cite[\S 6]{higson-k-theory-stable} and Corti\~{n}as-Thom for rings $A\otimes_\mathbf{C} I$ with an H-unital $C^*$-algebra $A$ and a sub-harmonic ideal $I$ satisfying $I=[I,I]$ \cite[Thm.~8.2.5]{cortinas-thom-comparison} as well as for rings $S\otimes_\mathbf{C} C$ with a smooth $\mathbf{C}$-algebra $S$ and a commutative $C^*$-algebra $C$ \cite[Thm.~1.5]{cortinas-thom-acta}.
%commutative $C^*$-algebras \cite[Thm.~3.1]{rosenberg-k-regularity}, 
Recently, Aoki showed the case of continuous functions on compact Hausdorff spaces with values in local division rings \cite{aoki}.
Of course, this list of references is not exhaustive.

So far, it has been an open problem whether the K-theory of a non-coherent regular ring satisfies homotopy invariance, e.g. consider \cite[p.~8198]{braunling-frobenius}. Showing that semi-valuation rings with non-trivial valuation are regular (Lemma~\ref{semi-valuation-ring-regular--lem}) we add to this line of research the following.

\begin{remark}[Remark~\ref{k-regularity--rem}]
Certain semi-valuation rings give examples of non-coherent regular rings whose algebraic K-theory does not satisfy homotopy invariance.
\end{remark}

By private communication, the author also learned of another such example found by Luca Passolunghi.

%Given a ring $A^+$ such that for all $k\geq 1$ the canonical map $\K(A^+)\to\K(A^+[t_1,\ldots,t_k])$ is an equivalence, one can ask whether the ring
%Building on a conjecture of Vorst \cite{vorst} one can ask whether

\vspace{6pt}
Semi-valuation rings have been introduced by Temkin \cite{temkin-relative-zr} and they occur as the stalks of his relative Riemann-Zariski spaces $\RZ_Y(X)$ which are locally ringed spaces associated with any separated morphism $f\colon Y\to X$ between quasi-compact and quasi-separated (qcqs) schemes. Furthermore, semi-valuation rings are important since they are the stalks of discretely ringed adic spaces, see Remark~\ref{local-Huber-pairs--rem}.
Exploiting that algebraic K-theory commutes with the formation of stalks, we deduce the following statement for the K-theory sheaves $\KZR$ and $\KH^\mathrm{RZ}$ on this space, see Definition~\ref{KZR--def} for a precise definition.

\begin{thm*}[Proposition~\ref{KZR-homotopy-invariant-U-noetherian--prop}, Corollary~\ref{generalisation--cor}]
Let $f\colon Y\to X$ be a separated morphism between qcqs schemes. Assume that $Y$ is of finite dimension, that all its stalks are stably coherent regular rings, and that the morphism $f\colon Y\to X$ admits a compactification (e.g.\ $f$ is of finite type). Then the canonical maps
\begin{enumerate}
\item $\KZR(-) \To \KZR((-)[t_1,\ldots,t_k])$ and
\item $\KZR(-) \To \KH^\mathrm{RZ}(-)$
\end{enumerate}
are equivalences of spectrum-valued sheaves on the topological space $\RZ_Y(X)$.
\end{thm*}

In the special setting that the morphism $Y\to X$ is the immersion of a regular dense open subscheme $Y$ into a divisorial noetherian scheme $X$ and assuming $k=1$, this specialises to the author's previous result for the K-theory of admissible Zariski-Riemann spaces $\skp{X}_Y$ \cite{k-admissible-zr}, see Corollary~\ref{generalisation--cor} and Remark~\ref{comparison--rem}. %But here we do not have to impose these assumptions on $X$.

\vspace{6pt}\noindent\textbf{Notation.} 
Discrete categories are denoted by $\mathrm{upright}$ $\mathrm{letters}$ whereas genuine $\infty$-categories are denoted by $\mathbf{bold}$ $\mathbf{letters}$.
For a ring $R$ we denote by $\K(R)$ its non-connective algebraic K-theory spectrum à la Blumberg-Gepner-Tabuada \cite[\S 9.1]{bgt13} which is an object of the $\infty$-category $\Sp$ of spectra \cite[1.4.3.1]{ha}; its associated object in the homotopy category of spectra is equivalent to the K-theory spectrum of Thomason-Trobaugh \cite[Def.~3.1]{tt90}. In particular, the connective part $\K_{\geq 0}(R)$ is equivalent to Quillen's K-theory \cite{quillen73}, combine \cite[Prop.~3.10]{tt90} with \cite[\S 7.2]{bgt13}.

\vspace{6pt}\noindent\textbf{Acknowledgements.}
This work was supported by the Deutsche Forschungsgemeinschaft (DFG) through the Collaborative Research Centre TRR 326 \textit{Geometry and Arithmetic of Uniformized Structures}, project number 444845124. I thank Katharina Hübner for helpful discussions around the coherence of semi-valuation rings as well as Oliver Bräunling, Matthew Morrow, and Georg Tamme for conversations around this paper's content.

%%%%%%%%%%%%%%%%%%%%%%%%%%%%%%%%%%%%%%%%%%%%%%%%%%%%%%%%%%%%%%%%%%%%%%%
\section{Semi-valuation rings}
\label{section--semi-valuation-rings}
%%%%%%%%%%%%%%%%%%%%%%%%%%%%%%%%%%%%%%%%%%%%%%%%%%%%%%%%%%%%%%%%%%%%%%%

For the convenience of the reader we recollect the definition and some basic facts about semi-valuation rings. All this is due to Temkin \cite{temkin-relative-zr}. Let us repeat some relevant terminology:
A \emph{valuation} on a ring $R$ is a map $|-|\colon R\to\Gamma\cup\{0\}$ for a totally ordered multiplicative abelian group $\Gamma$ such that $|1|=1$, $|xy|=|x|\cdot|y|$, and $|x+y| \leq \max\{|x|,|y|\}$ for all $x,y\in R$ with the convention that $0\cdot\gamma=0$ and $0<\gamma$ for all $\gamma\in\Gamma$. It follows that $|0|=0$ and that the support $\mathrm{supp}(|-|) := |-|\inv(0)$ is a prime ideal. Two valuations $|-|_1$ and $|-|_2$ on $R$ are called \emph{equivalent} if for all $x,y\in R$ the conditions $|x|_1 \leq |y|_1$ and $|x|_2 \leq |y|_2$ are equivalent. 

\begin{defin}[Temkin {\cite[\S 2.1]{temkin-relative-zr}.}] \label{semi-valuation-ring--def}
A \df{semi-valuation ring} is a pair $(A^+,|-|)$ consisting of a ring $A^+$ and a valuation $|-|\colon A^+ \to \Gamma\cup\{0\}$ such that
\begin{enumerate}
	\item every zero divisor of $A^+$ lies in the kernel of $|-|$ and
	\item for all $x,y\in A^+$ with $|x| \leq |y| \neq 0$ one has $y|x$.
\end{enumerate}
If $A^+$ is a semi-valuation ring and $\pfrak := \mathrm{supp}(|-|)$, then the local ring $A^+_\pfrak$ is called its \df{semi-fraction ring} and the ring $A^+/\pfrak$ is a valuation ring. We call $\pfrak$ the \df{valuative ideal} of $A^+$.
\end{defin}

We have the following characterisation of semi-valuation rings.

\begin{lemma} \label{semi-valuation-rings-characterisation--lem}
For a semi-valuation ring $A^+$ with valuative ideal $\pfrak$, semi-fraction ring $A:=A^+_\pfrak$, valuation ring $V:=A^+/\pfrak$, and residue field $k:=A^+_\pfrak/\pfrak$ the induced square
\begin{align*} \tag{$\square$} \begin{xy}\xymatrix{
A^+ \ar[r]\ar[d]_\pi & A \ar[d]\\ V \ar[r]& k
}\end{xy} \end{align*}
is a Milnor square (i.e.\ a bicartesian square of rings where two parallel arrows are surjective). In particular, $A^+$ is a local ring with maximal ideal $\mfrak^+ := \pi\inv(\mfrak_V)$.

Consequently, the following data are equivalent:
\begin{enumerate}
\item A semi-valuation ring $(A^+,|-|)$ up to equivalence of the valuation $|-|$.
\item A local ring $(A^+,\mfrak^+)$ and a prime ideal $\pfrak\subset A^+$ such that $A^+/\pfrak$ is a valuation ring.
\item A local ring $(A,\pfrak)$ and a valuation ring $V\subseteq A/\pfrak$.
\end{enumerate}
\end{lemma}
\begin{proof}
The implications (i) $\Rightarrow$ (ii) $\Leftrightarrow$ (iii) are clear. Assuming (ii), we obtain a valuation $|-|\colon A^+ \to[\pi] V \to V^\times/k^\times\cup\{0\}$ since $V$ is assumed to be a valuation ring. We will check the conditions (i) and (ii) of Definition~\ref{semi-valuation-ring--def}. As $V$ is a domain, we get (i). Now let $x,y\in A^+$ with $|x|\leq |y|\neq 0$, hence $y\notin\pfrak$. Since the morphism $A\to k$ reflects units, we get that $y\in A^\times$ and that for every $p\in\pfrak$ the ratio $\frac{p}{y}\in A$ lies in $A^+$. 
%In particular, if $|x|=0$, i.e.\ $x\in\pfrak$, then $x=y\cdot\frac{x}{y}$ in $A^+$ as desired. Now assume that $|x|\neq 0$, i.e.\ $x\notin\pfrak$.
By standard properties of valuation rings we find an element $z\in A^+$ such that $\pi(x)=\pi(y)\cdot\pi(z)$ in $V$, hence $x=yz+p$ for a suitable element $p\in\pfrak$. It follows that $x=yz+p=y(z+\frac{p}{y})$ in $A^+$ as desired.
\end{proof}

\begin{notation}
In the sequel, we refer to a semi-valuation ring as a pair $(A^+,\pfrak)$ as in Lemma~\ref{semi-valuation-rings-characterisation--lem}, even though the valuation is only defined up to equivalence. We always assume the square $(\square)$ to be implicitly defined.
\end{notation}

\begin{lemma} \label{semi-valuation-ring-maximal-ideal--lem}
Let $(A^+,\pfrak)$ be a semi-valuation ring with a non-trivial valuation (equivalently, its valuation ring $V=A^+/\pfrak$ is not a field). Then $\colim_{x\in \mfrak^+\setminus\pfrak}\,(x)$ is a \emph{filtered} colimit and its canonical morphism to $\mfrak^+$ is an isomorphism of ideals of $A^+$.
\end{lemma}
\begin{proof}
Since the valuation is assumed to be non-trivial, the set $\mfrak^+\setminus\pfrak$ is non-empty. According to condition (ii) in Definition~\ref{semi-valuation-ring--def}, the elements of $I$ are totally ordered with respect to divisibility, hence the colimit is filtered. This implies the claim since every module is the filtered colimit of its finitely generated submodules.
\end{proof}

\begin{remark} \label{local-Huber-pairs--rem}
Semi-valuation rings are precisely the rings $A^+$ occuring in \emph{local Huber pairs} $(A,A^+)$ as defined by Hübner-Schmidt \cite[p.~407f.]{huebner-schmidt} which are the local rings for discretely ringed adic spaces \cite[10.9~(i)]{huebner-schmidt}.
%Fujiwara-Kato deal with the notion of an \emph{$I$-valuative} ring where $I$  is any finitely generated ideal generating the topology on the valuation ring $\O/\ker(|-|)$ \cite[ch.~0, \S 8.7]{fuji-kato}.}
\end{remark}

%%%%%%%%%%%%%%%%%%%%%%%%%%%%%%%%%%%%%%%%%%%%%%%%%%%%%%%%%%%%%%%%%%%%%%%
\section{Coherence and regularity of semi-valuation rings}
%%%%%%%%%%%%%%%%%%%%%%%%%%%%%%%%%%%%%%%%%%%%%%%%%%%%%%%%%%%%%%%%%%%%%%%

In this section, we recall the notions of coherence and regularity for rings. For an elaborate treatment we refer to the book of Glaz \cite{glaz}. Afterwards, we will examine these notions for semi-valuation rings.

\begin{defin} \label{coherent--def}
Let $A$ be a ring. An $A$-module $M$ is said to be\ldots 
\begin{enumerate}
\item \df{coherent} if every finitely generated submodule is finitely presented. The ring $A$ is said to be coherent if it is a coherent $A$-module.
\item \df{finitely $n$-presented} for $n\in\N$ if there exists an exact sequence
\[
F_n \To \ldots \To F_2 \To F_1 \To F_0 \To M \To 0
\]
with finitely generated free $A$-modules $F_0,\ldots,F_n$.
\item \df{pseudo-coherent} if there exists an exact sequence
\[
\ldots \To P_2 \To P_1 \To P_0 \To M \To 0
\]
with finitely generated projective $A$-modules $(P_i)_{i\in\N}$.
%\item \df{perfect} if there exists an exact sequence
%\[
%0\To P_n \To \ldots \To P_2 \To P_1 \To P_0 \To M \To 0
%\]
%with finitely generated projective $A$-modules $P_0,\ldots,P_n$ for some $n\in\N$.
\end{enumerate} 
\end{defin}

\noindent It follows immediately that a ring is noetherian if and only if every finitely generated module is coherent and that a ring is coherent if and only if every finitely presented module is coherent.

For regularity, we have several notions appearing in the literature which are slightly different, see Remark~\ref{regularity--remark} below.

\begin{defin} \label{regular--def}
Let $A$ be a ring and let $n\in\N$. We say that $A$ is \ldots
\begin{enumerate}
\item \df{regular} if every pseudo-coherent module has finite projective dimension.
\item $n$-\df{regular} if every finitely $n$-presented $A$-module has finite projective dimension.
\item \df{uniformly regular} if the projective dimensions of all pseudo-coherent modules are uniformly bounded.
\item \df{uniformly $n$-regular} if the projective dimensions of all finitely $n$-presented modules are uniformly bounded.
\item \df{Glaz-regular} if every finitely generated ideal of $A$ has finite projective dimension.
%\item \df{weakly regular} if the projective dimensions of all finitely presented modules are uniformly bounded.
\end{enumerate} 
\end{defin}

\begin{prop} \label{regularity--prop}
For any ring, we have the following implications:
\begin{enumerate}
\item For all $n\geq 0\colon$ (uniformly) $n$-regular $\Rightarrow$ (uniformly) $(n+1)$-regular.
\item For all $n\geq 0\colon$ uniformly $n$-regular $\Rightarrow$ $n$-regular.
\end{enumerate}
For local rings and any $n\geq 0$, we have the implication:
\begin{enumerate}\setcounter{enumi}{2}
\item $n$-regular $\Rightarrow$ regular
\end{enumerate}
For a coherent ring $A$, the following are equivalent:
\begin{enumerate}\setcounter{enumi}{3}
\item $A$ is regular.
\item $A$ is $n$-regular for some $n\geq 1$.
\item $A$ is $n$-regular for all $n\geq 1$.
\item $A$ is Glaz-regular.
\end{enumerate}
\end{prop}
\begin{proof} 
The implications of (i) and (ii) hold by design. The implication in (iii) follows since projective modules over local rings are free. 
The implication (vi)$\Rightarrow$(v) is trivial. If $A$ is coherent, every finitely presented $A$-module is pseudo-coherent and vice versa, hence (iv)$\Leftrightarrow$(v).
The equivalence with (vii) goes by induction on the number of generators \cite[Thm.~6.2.1]{glaz}.
\end{proof}

\begin{remark} \label{regularity--remark}
The notion ``1-regular'' is called ``regular'' by Gersten \cite[Def.~1.3]{gersten-free-rings}. 
The notion ``Glaz-regular'' is called ``regular'' by Glaz \cite[ch.~6, \S 2]{glaz}.
For a coherent ring, these notions agree with Waldhausen's notion ``regular coherent'' \cite[p.~138]{waldhausen-free-1}.
Antieau-Mathew-Morrow \cite[\S 2]{antieau-mathew-morrow} call a ring ``weakly regular'' if it has finite flat dimension (for them a regular ring is noetherian). For coherent rings, the notion ``weakly regular'' is equivalent to the notion ``uniformly 1-regular'' and it is \emph{stronger} than the notions used by Gersten and Glaz.
For arbitrary rings, the notion ``regular'' from Definition~\ref{regular--def} seems to be most meaningful (at least to the author).
\end{remark}

%%%%%%%%%%%%%%%%%%%%%%%%%%%%%%%%%%%%%%%%%%%%%%%%%%%%%%%%%%%%%%%%%%%%%%%
\subsection*{Regularity of semi-valuation rings}
%%%%%%%%%%%%%%%%%%%%%%%%%%%%%%%%%%%%%%%%%%%%%%%%%%%%%%%%%%%%%%%%%%%%%%%

\begin{lemma} \label{projective-dimension--lem}
Let $(R,\mfrak)$ be a local ring and $M$ an $R$-module.
\begin{enumerate}
\item If $M$ is finitely presented and $\Tor_1^R(M,R/\mfrak)=0$, then $M$ is free. 
\item If $M$ is finitely $(n+1)$-presented and $\Tor_{n+1}^R(M,R/\mfrak)=0$ for some $n\geq 0$, then $M$ has projective dimension $\leq n$.
\end{enumerate} 
\end{lemma}
\begin{proof}
We follow the standard proof \cite[ch.~IV, Thm.~8]{serre-local-algebra} where the noetherian hypothesis is stated, but not needed.

\noindent (i) Let $x_1,\ldots,x_n\in M$ such that their images form a $R/\mfrak$-basis of $M/\mfrak M$ so that the morphism $\phi \colon R^n\to M, e_i\mapsto x_i,$ induces an isomorphism $\bar{\phi} \colon (R/\mfrak)^n \to M/\mfrak M$. By Nakayama's lemma we get an exact sequence $0\to N\to R^n\to[\phi] M \to 0$, inducing an exact sequence
\[
0 = \Tor_1^R(M,R/\mfrak) \To N/\mfrak N \To (R/\mfrak)^n \To[\cong] M/\mfrak M \To 0
\]
so that $N/\mfrak N=0$. If $M$ is finitely presented, then $N$ is finitely generated, hence $N=0$ by Nakayama's lemma.

\noindent (ii)  Let
\[
F_{n+1} \To[f_{n+1}] F_n \To[f_n] \ldots \To F_1 \To[f_1] F_0 \To[f_0] M \To 0
\]
be an exact sequence where $F_{n+1},\ldots,F_0$ are finitely generated free. Setting $K_i := \ker(f_i)$ we get that
\[
\Tor^R_1(K_{n-1},R/\mfrak) = \Tor^R_2(K_{n-2},R/\mfrak) = \ldots = \Tor^R_n(K_0,R/\mfrak) = \Tor^R_{n+1}(M,R/\mfrak) =0.
\]
by assumption. Since $M$ is finitely $(n+1)$-presented, $K_{n-1}$ is finitely presented, hence free by (i), so that $M$ has projective dimension $\leq n$. 
\end{proof}

\begin{example} \label{example-svr-regular}
Let $A:=\Q_p[[X,Y]]$ and let $A^+$ be defined to be the pullback in the Milnor square
\carremap{A^+}{}{\pi}{A}{\mathrm{ev}_{0,0}}{\Z_p}{}{\Q_p.}
Then $A^+$ has maximal ideal $\mfrak^+ = \pi\inv(p\Z_p) = (p,X,Y) = (p)$ as  $X=p\cdot\frac{X}{p}, Y=p\cdot\frac{Y}{p}\in (p)$. Thus $\Tor_2^{A^+}(M,A^+/\mfrak^+)=0$ for any $A^+$-module $M$, so that every finitely 2-presented $A^+$-module has projective dimension $\leq 1$ by Lemma~\ref{projective-dimension--lem}. Thus $A^+$ is 2-regular.
\end{example}

We can generalise this example to any semi-valuation ring with non-trivial valuation.

\begin{lemma} \label{semi-valuation-ring-regular--lem}
Let $(A^+,\pfrak)$ be a semi-valuation with non-trivial valuation. Then every finitely 2-presented $A^+$-module has projective dimension $\leq 1$. In particular, the ring $A^+$ is 2-regular.
\end{lemma}
\begin{proof}
By Lemma~\ref{semi-valuation-ring-maximal-ideal--lem}, we can write $\mfrak=\colim_{x\in I}(x)$ as a filtered colimit with $I=\mfrak^+\setminus\pfrak\neq\emptyset$. According to condition (i) in Definition~\ref{semi-valuation-ring--def}, the sequence $0\to A^+\to[\cdot x] A^+\to A^+/(x)\to 0$ is exact for every $x\in I$. Since tensor products are cocontinuous in each variable and since filtered colimits are exact, we get that
\[ \Tor^{A^+}_2(M,A^+/\mfrak^+) = \colim_{x\in I}\,\Tor^{A^+}_2(M,A/(x)) = 0 \]
for any $A^+$-module $M$. By Lemma~\ref{projective-dimension--lem}, every finitely 2-presented $A^+$-module has projective dimension $\leq 1$ as desired.
\end{proof}

%%%%%%%%%%%%%%%%%%%%%%%%%%%%%%%%%%%%%%%%%%%%%%%%%%%%%%%%%%%%%%%%%%%%%%%
\subsection*{Regularity in Milnor squares}
%%%%%%%%%%%%%%%%%%%%%%%%%%%%%%%%%%%%%%%%%%%%%%%%%%%%%%%%%%%%%%%%%%%%%%%

More generally than in the situation of semi-valuation rings, one can wonder under which circumstances pullback rings in Milnor squares are regular. Let
\begin{align*} \tag{M} \begin{xy}\xymatrix{
A \ar[r]^f \ar[d]_p & B \ar[d]^q\\ A' \ar[r]^{f'}& B'
}\end{xy} \end{align*}
be a Milnor square (i.e. a cartesian square of rings such that $p$ and $q$ are surjective). It follows that the square (M) is also a pushout square, i.e.\ $B'\cong B\otimes_AA'$. 

\begin{example}
Let $k$ be field of characteristic $\neq 2$ and consider the node $A=k[X,Y]/(Y^2-X^3-X^2)$. It fits into a Milnor square (M) with $B=k[T]$, $f$ being the normalisation (i.e.\ $X\mapsto T^2-1$ and $Y\mapsto T(T^2-1)$), and $A'=A/(X,Y)$ being the origin. Then $A'$, $B$, and $B'$ are regular noetherian rings, but $A$ is not regular.
\end{example}

\begin{question} \label{milnor-square-regularity--lem}
Given a Milnor square (M) such that the morphism $A\to[f] B$ is of finite Tor-dimension and assuming that the rings $A'$, $B$, and $B'$ all are regular. Is then the ring $A$ regular?
\end{question}

A helpful result for answering the question might be the following statement.

\begin{lemma} \label{perfect-module-milnor-square--lem}
Let 
\carre{A}{B}{A'}{B'}
be a pullback square of $\mathbf{E}_1$-ring spectra. Then an $A$-module $M$ is perfect if and only if the base changes $M\otimes^\mathrm{L}_AB$ and $M\otimes^\mathrm{L}_AA'$ both are perfect.
\end{lemma}
\begin{proof} \newcommand{\lax}{\overset{\rightarrow}{\times}}
If an $A$-module is perfect, this also holds true for any base change. For the other implication we use that there exists a commutative square
\[ \begin{xy} \xymatrix{
	\Perf(A) \ar@{^{(}->}[r] \ar@{^{(}->}[d] & \Perf(B)\lax_{\Perf(B')}\Perf(A') \ar@{^{(}->}[d] \\
	\RModinf(A) \ar@{^{(}->}[r] & \RModinf(B)\lax_{\RModinf(B')}\RModinf(A')
} \end{xy} \]
where $\RModinf(-)$ denotes the presentable, stable $\infty$-category of (derived) right modules, $\Perf(-)$ its full subcategory of perfect modules (which are precisely the compact objects), and ``$\lax$'' the lax pullback \cite[Def.~5]{tamme-excision}. The horizontal functors (which are induced by the base change functors) are fully faithful \cite[1.7]{land-tamme} and  the left vertical functor is the inclusion of compact objects by definition. The right vertical functor identifies also with the inclusion of compact objects \cite[Prop.~13]{tamme-excision}. Now the claim follows since base change, and hence the horizontal functors, preserve filtered colimits \cite[Lem.~8~(iii)]{tamme-excision}.
\end{proof}

%%%%%%%%%%%%%%%%%%%%%%%%%%%%%%%%%%%%%%%%%%%%%%%%%%%%%%%%%%%%%%%%%%%%%%%
\subsection*{Non-coherence of semi-valuation rings}
%%%%%%%%%%%%%%%%%%%%%%%%%%%%%%%%%%%%%%%%%%%%%%%%%%%%%%%%%%%%%%%%%%%%%%%

\begin{lemma} \label{not-coherent--lem}
Let $(A^+,\pfrak)$ be a semi-valuation ring such that $A=A^+_\pfrak$ is not finitely generated as a module over $A^+$ and such that there exists a regular sequence $X,Y$ in $\pfrak$. Then the ring $A^+$ is not coherent.
\end{lemma}
\begin{proof}
Consider the homomorphism $\phi \colon A^+\times A^+ \to A^+, (f,g)\mapsto fX-gY$. We see immediately that $A\cdot (Y,X) \subseteq \ker(\phi)$. On the other hand, for $(f,g)\in\ker(\phi)$ we get $gY=0$ in $A^+/(X)$ by regularity so that $g\in X\cdot A^+$. By symmetry, $f\in Y\cdot A^+$. Setting $g=g'X$ and $f=f'Y$ (for suitable elements $f'$ and $g'$ in $A^+$) we get $0=(f'-g')XY$, hence $f'=g'$ so that $(f,g)=f'(Y,X)\in A\cdot(Y,X)$.
\end{proof}

\begin{example}
A concrete instance for the setting of Lemma~\ref{not-coherent--lem} is given by the semi-valuation ring of Example~\ref{example-svr-regular}, i.e.\ $A=\Q_p[[X,Y]]$ and $A^+=\{f\in A\,|\,f(0,0)\in\Z_p\}$ (and $X$ and $Y$ as themselves).
\end{example}

\begin{lemma} \label{not-coherent-2--lem}
Let $(A^+,\pfrak)$ be a semi-valuation ring with non-trivial valuation. If $A=A^+_\pfrak$ is coherent and not regular, then the ring $A^+$ is not coherent.
\end{lemma}
\begin{proof}
We know that the ring $A^+$ is 2-regular by Lemma~\ref{semi-valuation-ring-regular--lem}. If it was also coherent, then it would be Glaz-regular by Proposition~\ref{regularity--prop}, hence $A$ would be Glaz-regular \cite[6.2.3]{glaz}.
\end{proof}

\begin{example} \label{cusp-over-Q_p--ex}
An instance of the setting of Lemma~\ref{not-coherent-2--lem} comes from starting with any coherent local ring $(A,\pfrak)$ that is not regular and a non-trivial valuation on its residue field $k=A/\pfrak$, e.g. $A=\Q_p[X,Y]_{(X,Y)}/(Y^2-X^3)$ with the $p$-adic valuation on $k=\Q_p$, so that $V=\Z_p$.
\end{example}

\begin{remark}
Assuming that the ideal $\pfrak\subset A^+$ is a \emph{flat} $A^+$-module, then $A^+$ is a coherent ring provided that $A$ is a coherent ring \cite[Thm.~5.1.3]{glaz}.
\end{remark}

%%%%%%%%%%%%%%%%%%%%%%%%%%%%%%%%%%%%%%%%%%%%%%%%%%%%%%%%%%%%%%%%%%%%%%%
\section{K-theory and G-theory of semi-valuation rings}
\label{section--k-theory-semi-valuation-rings}
%%%%%%%%%%%%%%%%%%%%%%%%%%%%%%%%%%%%%%%%%%%%%%%%%%%%%%%%%%%%%%%%%%%%%%%

\begin{thm} \label{k-milnor-square--thm}
For a semi-valuation ring $(A^+,\pfrak)$ with semi-fraction ring $A=A^+_\pfrak$, valuation ring $V=A^+/\pfrak$, and field of fractions $k=A^+_\pfrak/\pfrak$ the induced square
\carre{\K(A^+[t_1,\ldots,t_r])}{\K(A[t_1,\ldots,t_r])}{\K(V[t_1,\ldots,t_r])}{\K(k[t_1,\ldots,t_r])}
of non-connective algebraic K-theory spectra is cartesian for every $r\in\N$.
\end{thm}
\begin{proof}
Let $\F$ be the prime field of $k$ so that the valuation ring $V\cap\F$ has rank $\leq 1$. Let $I$ be the set of all subextensions $\ell\subset k$ that have finite transcendence degree over $\F$. Note that $k \cong \colim_{\ell\in I} \ell$ within the category of rings. Now let $\ell\in I$. By the ``dimension inequality'' \cite[Thm.~3.4.3, Cor.~3.4.4]{engler-prestel}, the valuation ring $V_\ell := V\cap\ell$ has finite rank. Hence $V$ is a filtered colimit of valuation rings of finite rank. Set $A_\ell$ to be the preimage of $\ell$ in $A$ which is a local ring with residue field $\ell$. Hence the preimage $A^+_\ell$ of $V$ in $A_\ell$ is a semi-valuation ring and the square $(\square)$ is the filtered colimit of its restrictions along $\ell\inj k$ for all $\ell\in I$.
 
Since K-theory commutes with filtered colimits of rings we may assume that $V$ has finite rank, hence it is a microbial valuation ring (i.e. has a prime ideal of height 1). In this case, there exists an element $s\in A^+$ such that $V[\bar{s}\inv]=K$ where $\bar{s}=\pi(s)$ for the projection $\pi\colon A^+\to V$. Then one checks easily that $s\in A^\times$, $A^+[s\inv]= A$, and $\pfrak \subseteq s^n\cdot A^+$ for all $n\in\N$. Hence the map $\pi\colon A^+\to V$ is an analytic isomorphism along $S=\{s^n\,|\,n\in\N\}$ and the same holds true for the induced maps $A^+[t_1,\ldots,t_r] \to V[t_1,\ldots,t_r]$ for all $r\in\N$. By Weibel's analytic isomorphism theorem \cite[Thm.~1.3]{weibel-analytic} the square of Theorem~\ref{k-milnor-square--thm} is cartesian.
Note that \emph{every} Milnor square induces a cartesian square on non-positive K-theory $\K_{\leq 0}$ \cite[Ch.~XII, Thm.~(8.3), p.~677]{bass68}.
\end{proof}

\begin{cor}
Let $(A^+,\pfrak)$ be a semi-valuation ring. Then for $n<0$ the canonical morphism $\K_n(A^+)\to\K_n(A)$ is an isomorphism. If $A=A^+_\pfrak$ is noetherian of finite dimension $d$, then 
\begin{enumerate}
\item $\K_n(A^+)=0$ for $n<-\dim(A)$,
\item $\K_n(A^+)\To[\simeq]\K_n(A^+[t_1,\ldots,t_k])$ for $n\leq -d$ and $k\geq 0$, and
\item $\K_{-d}(A^+) \cong \Ho_\cdh^d(\Spec(A),\Z)$.
\end{enumerate}
\end{cor}
\begin{proof}
This follows from Theorem~\ref{k-milnor-square--thm}, the vanishing of negative K-theory of polynomial rings over valuation rings, and the corresponding statements for the ring $A$ which hold due to a result of Kerz-Strunk-Tamme \cite[Thm.~B, Thm.~D]{kst16}.
\end{proof}

\begin{cor} \label{K-semi-valuation-ring--cor}
Let $(A^+,\pfrak)$ be a semi-valuation ring whose semi-fraction ring $A=A^+_\pfrak$ is stably coherent and regular. Then the canonical maps
\begin{enumerate}
\item $\K_{\geq 0}(A^+) \To \K(A^+)$, 
\item $\K(A^+) \To \K(A^+[t_1,\ldots,t_k])$, and
\item $\K(A^+) \To \KH(A^+)$
\end{enumerate}
are equivalences.
\end{cor}
\begin{proof}
The claimed equivalences (i)-- (iii) follow formally from Theorem~\ref{k-milnor-square--thm} and by the known corresponding statements for the K-theory of stably coherent regular rings \cite{swan-coherent} and valuation rings \cite{waldhausen-free-1}, see also Kelly-Morrow \cite[Thm.~3.3]{kelly-morrow}.
\end{proof}

\begin{remark}[Regularity does not imply K-regularity] \label{k-regularity--rem}
We say that a ring $A$ is \textbf{K-regular} if for every $k\geq 1$ the canonical map $\K(A)\to\K(A[t_1,\ldots,t_k])$ is an equivalence of spectra. From Theorem~\ref{k-milnor-square--thm} we deduce that a semi-valuation ring $A^+$ is K-regular if and only if its semi-fraction ring $A$ is K-regular. Given a coherent local ring $A$ which is \emph{not} K-regular together with a non-trivial valuation on its residue field (e.g.\ the ring $A$ in Example~\ref{cusp-over-Q_p--ex}), then the associated semi-valuation ring $A^+$ is a non-coherent (Lemma~\ref{not-coherent-2--lem}) and regular (Lemma~\ref{semi-valuation-ring-regular--lem}) ring which which is \emph{not} K-regular.
\end{remark}

\begin{remark}[G-theory]
Let $A$ be a ring. Denote by $\PCoh(A)$ the full subcategory of $\Mod(A)$ spanned by pseudo-coherent modules (Definition~\ref{coherent--def}); it is an exact subcategory \cite[II.7.1.4]{weibel}.
The \df{G-theory} of $A$ is defined as
\[
 \Ge(A) := \K(\PCoh(A))
\]
where $\K$ denotes Schlichting's \emph{non-connective} K-theory for exact categories \cite{schlichting04}; c.f.\ Thomason-Trobaugh \cite[3.11.1]{tt90} and Weibel's \emph{K-book} \cite[V.2.7.4]{weibel}; note that this can also be realised as the K-theory of a stable $\infty$-category, see \cite[\S 8]{henrard}.
If a ring $A$ is regular, then the canonical map $\K(A)\to\Ge(A)$ is an equivalence, since -- by definition of regularity -- the inclusion $\Vect(A)\inj\PCoh(A)$ satisfies the conditions of the resolution theorem \cite[V.3.1]{weibel}.
\end{remark}

%%%%%%%%%%%%%%%%%%%%%%%%%%%%%%%%%%%%%%%%%%%%%%%%%%%%%%%%%%%%%%%%%%%%%%%
\section{K-theory of relative Riemann-Zariski spaces}
\label{section--k-theory-relative-ZR-spaces}
%%%%%%%%%%%%%%%%%%%%%%%%%%%%%%%%%%%%%%%%%%%%%%%%%%%%%%%%%%%%%%%%%%%%%%%

In this section we generalise the results from the author's previous article on the K-theory of admissible Zariski-Riemann spaces \cite{k-admissible-zr} to the setting of Temkin's relative Riemann-Zariski spaces \cite{temkin-relative-zr};\footnote{The terms ``Zariski'' and ``Riemann'' appear in different orders in the literature, cf.\ \cite{temkin-relative-zr} vs. \cite{kst}, and the author tries to be coherent with these sources.} the former are defined for the inclusion of an open subscheme whereas the latter are defined for an arbitrary separated morphism. The statements are reduced to the stalks of these spaces which are semi-valuation rings so that we can use the results from section~\ref{section--k-theory-semi-valuation-rings}.

\begin{notation*}
In this section let $f\colon Y\to X$ be a separated morphism between quasi-compact and quasi-separated schemes.
\end{notation*}

\begin{defin}[{Temkin \cite[\S 2.1]{temkin-relative-zr}}]
A $\mathbf{Y}$\df{-modification} of $X$ is a factorisation $Y\to[g_i] X_i \to[f_i] X$ of $f$ into a schematically dominant morphism $g_i$ and a proper morphism $f_i$. We denote by $\Mdf_Y(X)$ the category of $Y$-modifications of $X$ together with compatible morphisms. The \df{relative Riemann-Zariski space} $\RZ_Y(X)$ is the limit $\lim_iX_i$ within the category of locally ringed spaces, indexed by the cofiltered category $\Mdf_Y(X)$.
\end{defin}

\begin{lemma} \label{semi-valuation-relative-zr--lem}
Let $(A^+,\pfrak)$ be a semi-valuation ring with semi-fraction ring $A=A^+_\pfrak$. Then the canonical projection
	\[
	\RZ_{\Spec(A)}(\Spec(A^+)) \To \Spec(A^+)
	\]
is an isomorphism. 
\end{lemma}
\begin{proof}
We have a bicartesian square
\carre{\Spec(A^+_\pfrak/\pfrak)}{\Spec(A)}{\Spec(A^+/\pfrak)}{\Spec(A^+).}
Every $\Spec(A)$-modification $X\to\Spec(A^+)$ yields a lift $\Spec(A^+/\pfrak)\to X$ by the valuative criterion of properness. Hence we get a section $\Spec(A^+)\to X$ so that $\id_{\Spec(A^+)}$ is cofinal in $\Mdf_{\Spec(A)}(\Spec(A^+))$.
\end{proof}

\begin{cor}
Let $A$ and $A^+$ be as in Lemma~\ref{not-coherent--lem}. Then the relative Riemann-Zariski space $\RZ_{\Spec(A)}(\Spec(A^+))$ is not cohesive (i.e.\ its structure sheaf is not coherent over itself).
\end{cor}
\begin{proof}
This follows from Lemma~\ref{not-coherent--lem} together with Lemma~\ref{semi-valuation-relative-zr--lem}.
\end{proof}

The following corollary answers the question whether admissible Zariski-Riemann spaces are cohesive, see \cite[Prop.~3.10]{k-admissible-zr} and the preceding paragraph in loc.\ cit..

\begin{cor}
Let $X$ be a quasi-compact and quasi-separated scheme, $U\subseteq X$ a quasi-compact open subscheme, $\skp{X}_U$ the associated admissible Zariski-Riemann space, and $i\colon \tilde{Z} \inj \skp{X}_U$ be the inclusion of the closed complement with reduced structure. There exists an example of this situation such that the locally ringed space $\skp{X}_U$ is not cohesive.
\end{cor}
\begin{proof}
Consider the rings $A:=\Q_p[[X,Y]]$ and $A^+=\{f\in A\,|\,f(0,0)\in\Z_p\}$ from Example~\ref{example-svr-regular}. Since $A=A^+[p\inv]$, the induced morphism $\Spec(A) \to \Spec(A^+)$ is an open immersion, so that $\skp{A^+}_A \cong \RZ_{\Spec(A)}(\Spec(A^+))$ is not cohesive.
\end{proof}

\begin{lemma} \label{ZR-stalks-are-semi-valuation-rings--lem}
For any point $x\in\RZ_Y(X)$, the stalk $\O_{\RZ_Y(X),x}$ is a semi-valuation ring.
\end{lemma}
\begin{proof}
This follows from \cite[Prop~2.2.1]{temkin-relative-zr} since the morphism $f\colon Y\to X$ is assumed to be separated which is equivalent to being decomposable by \cite[Thm.~1.1.3]{temkin-relative-zr}.
\end{proof}

\begin{defin} \label{KZR--def}
For an open subset $V$ of $\RZ_Y(X)$ we define the set $\Model(V)$ whose elements are open subsets $V'$ of some $X'\in\Mdf_Y(X)$ such that $p_{X'}\inv(V')=V$. Defining  $V'\leq V''$ if $V''=q\inv(V')$ for a morphism $q\colon X''\to X'$ in $\Mdf_Y(X)$ we get a partial order on $\Model(V)$. Since $\Mdf_Y(X)$ is cofiltered, the sets $\Model(V)$ are filtered posets. If $V$ is quasi-compact, then $\Model(V)$ is non-empty \cite[ch.~0, 2.2.9]{fuji-kato}. In particular, $\Mdf_Y(X)  = \Model(\RZ_Y(X))  $.
We define the \df{K-theory on the Riemann-Zariski space} to be the presheaf
	\[
	\KZR \,\colon\, \Open^\mathrm{qc}(\RZ_Y(X))\op \To \Sp,\quad V \mapsto \colim_{V'\in\Model(V)} \K(V'),
	\]
on the poset $\Open^\mathrm{qc}(\RZ_Y(X))$ of quasi-compact open subsets of $\RZ_Y(X)$. Analogously, we define the presheaf $\KH^\mathrm{RZ}(-)$.
\end{defin}

\begin{prop} \label{KZR-sheaf--prop}
The presheaves $\KZR(-)$ and $\KH^\mathrm{RZ}(-)$  are sheaves of spectra.  In particular, we get induced sheaves on the topological space $\RZ_Y(X)$.
\end{prop}
\begin{proof}
The topology on the category $\Open^\mathrm{qc}(\RZ_Y(X))$ equals the topology induced by the cd-structure of its cartesian squares. Hence it suffices to show that for any open subset $V$ of $\RZ_Y(X)$ which is covered by two open subsets $V_1,V_2\subset V$ with intersection $V_3 := V_1\cap V_2$ the induced square
	\carretag[$\spadesuit$]{\KZR(V)}{\KZR(V_1)}{\KZR(V_2)}{\KZR(V_3)}
is cartesian in $\Sp$. For $V_1'\in\Model(V_1)$ and $V_2'\in\Model(V_2)$ we may assume that they live on a common $Y$-modification of $X$. Then $V_1'\cup V_2' \in\Model(V)$ and $V_1'\cap V_2' \in\Model(V_3)$. By cofinality, the square ($\spadesuit$) is equivalent to a colimit of cartesian squares since K-theory is a Zariski-sheaf on the category of qcqs schemes. The statement for $ \KH^\mathrm{ZR}(-)$ has the same proof. Since the quasi-compact open subsets of $\RZ_Y(X)$ form a basis of the topology, the sheaves extend from $\Open^\mathrm{qc}(\RZ_Y(X))$ to $\Open(\RZ_Y(X))$.
\end{proof}

\begin{lemma} \label{KZR-stalks--lemma}
For $x\in\RZ_Y(X)$ the stalk of $\KZR_x$ is equivalent to $\K(\O_{\RZ_Y(X),x})$.
\end{lemma}
\begin{proof}
Since K-theory commutes with colimits, we can compute
\begin{align*}
\KZR_x 
& \simeq \colim_{x\in V}\KZR(V) \simeq \colim_{x\in V} \colim_{V'\in\Model(V)} \K(V') \\
& \simeq \colim_{X'\in\Mdf_Y(X)} \colim_{p_{X'}(x)\in V'} \K(\O_{X'}(V')) \\
& \simeq \colim_{X'\in\Mdf_Y(X)} \K(\O_{X',p_{X'}(x)}) \simeq \K(\O_{\RZ_Y(X),x})
\end{align*}
where the step from the first to the second line is due to cofinality of the indexing categories.
\end{proof}

\begin{prop} \label{KZR-homotopy-invariant-U-noetherian--prop}
Assume that all stalks of $Y$ are stably coherent regular rings and that $\Mdf_Y(X)$ admits a cofinal subcategory $\Mcal_d$ such that $\dim(X')\leq d$ for all $X'\in\Mcal_d$ for some $d\in\N$. Then the canonical maps
\begin{enumerate}
\item $\KZR(-) \To \KZR((-)[t_1,\ldots,t_k])$ and
\item $\KZR(-) \To \KH^\mathrm{ZR}(-)$
\end{enumerate}
are equivalences of spectrum-valued sheaves on $\RZ_Y(X)$.
\end{prop}
\begin{proof} 
For every $X'\in\Mcal_d$, the sheaf topos $\Sh(X')$ of space-valued sheaves on $X'$ has homotopy dimension $\leq d$ by a result of Clausen-Mathew \cite[3.12]{clausen-mathew}. This implies that $\Sh(\RZ_Y(X))$ has homotopy dimension $\leq d$ \cite[3.11]{clausen-mathew}. Analogously, every quasi-compact open subset $V$ of $\RZ_Y$ is a cofiltered limit of schemes of finite dimension so that $\Sh(V)$ has homotopy dimension $\leq d$. Thus $\Sh(\RZ_Y(X))$ is locally of homotopy dimension $\leq d$, hence Postnikov complete \cite[7.2.1.10]{htt}. Since the $\infty$-category $\Sh_\Sp(\RZ_Y(X))$ is equivalent to the category $\Sh_\Sp(\Sh(\RZ_Y(X)))$ \cite[1.3.1.7]{sag}, it is left-complete \cite[1.3.3.10,1.3.3.11]{sag}. Thus we can check equivalences of sheaves of spectra on $\RZ_Y(X)$ on stalks (this is folklore, see \cite[A.1.32]{phd}).

Hence we can check the statements (i) and (ii) on the stalks $\O_{\RZ_Y(X),x}$ which are semi-valuation rings (Lemma~\ref{ZR-stalks-are-semi-valuation-rings--lem}) so that the claim follows from Lemma~\ref{KZR-stalks--lemma} and Corollary~\ref{K-semi-valuation-ring--cor}.
\end{proof}

\begin{cor} \label{generalisation--cor}
Assume that $Y$ is of finite dimension, that all its stalks are stably coherent regular rings, and that the morphism $f\colon Y\to X$ admits a compactification (e.g.\ $f$ is of finite type). Then the properties (i) and (ii) of Proposition~\ref{KZR-homotopy-invariant-U-noetherian--prop} hold true.
\end{cor}
\begin{proof}
This follows since $\dim(X')\leq\dim(Y)$ for every $X'\in\Mdf_Y(X)$ by the assumptions.
\end{proof}

\begin{remark} \label{comparison--rem}
If the morphism $f$ is the inclusion of a schematically dense open subscheme $Y$ of $X$, the space $\RZ_Y(X)$ is isomorphic to the admissible Zariski-Riemann space $\skp{X}_Y$ \cite[2.7]{k-admissible-zr}.
In case that the morphism $f\colon Y\to X$ admits a compactification $Y \inj \bar{X} \to X$ (e.g.\ if $f$ is of finite type), then the canonical morphism $\skp{\bar{X}}_Y\cong\RZ_Y(\bar{X})\to\RZ_Y(X)$ is an isomorphism. If moreover $\bar{X}$ is noetherian, then Proposition~\ref{KZR-homotopy-invariant-U-noetherian--prop} with $k=1$ already follows from previous work of the author \cite[Cor.~4.18]{k-admissible-zr}. 
\end{remark}

%%%%%%%%%%%%%%%%%%%%%%%%%%%%%%%%%%%%%%%%%%%%%%%%%%%%%%%%%%
\bibliography{0-literature}
\bibliographystyle{amsalpha}
\end{document}